\documentclass[12pt]{amsart}

\usepackage{amssymb}

\usepackage{enumerate}

\usepackage{graphicx}

\makeatletter
\@namedef{subjclassname@2010}{%
  \textup{2010} Mathematics Subject Classification}
\makeatother

\newtheorem{thm}{Theorem}[section]

\newtheorem{lem}[thm]{Lemma}
\newtheorem{prop}[thm]{Proposition}

\theoremstyle{definition}

\numberwithin{equation}{section}

\def\Res{\operatorname{Res}}

\begin{document}


\baselineskip=17pt



\title[On the number of subgroups]{On the number of subgroups of the group $\mathbb{Z}_{m_{1}} \times \mathbb{Z}_{m_{2}}$ with $m_{1}m_{2}\leq x$ such that $m_{1}m_{2}$ is a $k$-th power}

\author[Y. Sui]{Yankun Sui}
\address{College of Information Engineering\\
Nanjing Xiaozhuang University\\
Nanjing 211171, P. R. China}
\email{yankun.sui.math@gmail.com}

\author[D. Liu]{Dan Liu}
\address{School of Science\\
Qingdao University of Technology\\
Qingdao 266525, P. R. China}
\email{liudan92chn@163.com}

\author[B. Zhou]{Boling Zhou}
\address{College of Information Engineering\\
Nanjing Xiaozhuang University\\
Nanjing 211171, P. R. China}
\email{bolingzhou09010@163.com}

\date{}

\footnotetext{


{This work is supported by National Natural Science Foundation of China (Grant No. 12301006, 12471009) and Beijing Natural Science Foundation (Grant No. 1242003).}
 }

\begin{abstract}
Let ${\Bbb Z}_{m}$ be the additive group of residue classes modulo $m$ and $s(m_{1},m_{2})$ denote
the number of subgroups of the group ${\Bbb Z}_{m_{1}}\times {\Bbb Z}_{m_{2}}$, where $m_{1}$ and $m_{2}$ are arbitrary positive integers.
We consider sums of type $\sum\limits_{\substack{m_{1}m_{2}\leq x \\ m_{1}m_{2}\in N_{k}}}s(m_{1},m_{2})$, where $N_{k}$ is the set of $k$-th power of natural numbers. In particular, we deduce asymptotic formulas with $k=2$ and $k=3$.
\end{abstract}

\maketitle

\section{\bf Introduction}

Let $\tau(n)$ be the number of divisors of $n$ and $\mathbb{Z}_{m}$ denote the additive group of residue classes modulo $m$. For arbitrary positive integers $m_{1}$ and $m_{2}$, consider the group $\mathbb{Z}_{m_{1}} \times \mathbb{Z}_{m_{2}}$. For more details about the group $\mathbb{Z}_{m_{1}} \times \mathbb{Z}_{m_{2}}$, see \cite{NT}. Let $s(m_{1},m_{2})$ denote the total number of subgroups of the group $\mathbb{Z}_{m_{1}} \times \mathbb{Z}_{m_{2}}$. Hampejs, Holighaus, T\'{o}th and Wiesmeyr \cite{H} studied the properties of the subgroups of the group $\mathbb{Z}_{m_{1}} \times \mathbb{Z}_{m_{2}}$ and derived a simple method to compute $s(m_{1},m_{2})$. For all $m_{1},m_{2} \in \mathbb{N}$, they showed that
\begin{equation}\label{s(m,n)}
s(m_{1},m_{2})=\sum_{d|m_{1},e|m_{2}}\gcd(d,e)\\
=\sum_{d|\gcd(m_{1},m_{2})}\phi(d)\tau(m_{1}/d)\tau(m_{2}/d),
\end{equation}
where $\gcd$ is the greatest common divisor, $\phi$ is Euler's totient function.

Nowak and T\'{o}th \cite{NT} took the lead in studing the average order of the function $s(m_{1},m_{2})$. They proved that
\begin{equation}\label{s m n}
\sum_{m_{1},m_{2}\leq x}s(m_{1},m_{2})=x^{2}\sum^{3}_{j=0}A_{j}\log^{j}x+O(x^{1117/701+\varepsilon}),
\end{equation}
where $A_{j}\ (0\leq j\leq 3)$ are explicit constants, whose definitions are omitted here. Later T\'{o}th and Zhai \cite{TZh} improved the error term in (\ref{s m n}) to $O(x^{1.5}(\log x)^{6.5})$.

More recently, Sui and Liu \cite{SL} considered the sum of $s(m,n)$ in the Dirichlet region $\{(m,n), mn\leq x\}$. Suppose $x>0$ is a real number, they proved that
\begin{equation}
\sum_{mn\leq x}s(m,n)=xP_{4}(\log x)+O(x^{2/3}\log^{6}x),
\end{equation}
where $P_{4}(u)$ is a polynomial in $u$ of degree 4 with the leading coefficients $1/(8\pi^{2})$.

In this paper, we focus on the following problem. Let $k\geq 2$ be a fixed integer and $N_{k}$ denote the set of $k$-th power of natural numbers. For arbitrary positive integers $m_{1}$ and $m_{2}$, we consider the number of subgroups of the group $\mathbb{Z}_{m_{1}} \times \mathbb{Z}_{m_{2}}$ with $m_{1}m_{2}\leq x$ such that $m_{1}m_{2}$ is a $k$-th power number.
For $x\geq1$, we define
\begin{align*}
    T_{k}(x):=\sum_{\substack{m_{1}m_{2}\leq x \\ m_{1}m_{2}\in N_{k}}}s(m_{1},m_{2}).
\end{align*}
The greater the value of $k$, the more complex the calculation. So, in this article, we first study two cases where the product of $m_{1}m_{2}$ is a square number and a cube number. The purpose of this paper is to obtain asymptotic formulas of $T_{2}(x)$ and $T_{3}(x)$, by using the complex integration method.

\ \

\begin{thm} \label{Th_1} For $k=2$, we have the asymptotic formula
\begin{equation}
T_{2}(x)=c_{2}x+O(x^{19/28+\varepsilon}),
\end{equation}
where $c_{2}$ is an absolute constant.
\end{thm}

\begin{thm} \label{Th_2} For $k=3$, we have the asymptotic formula
\begin{equation}
T_{3}(x)=x^{2/3}\sum_{j=0}^{4}a_{j}\log^{j}x+O(x^{5/8+\varepsilon}),
\end{equation}
where $a_{j} \ (j=0,1,\ldots,4)$ are absolute constants.
\end{thm}

{\bf Notation.} Throughout the paper, $\mathbb{N}$ denotes the set of all positive integers, $N_{k}$ is the set of $k$-th power of natural numbers, $\mathbb{Z}_{m}$ is the additive group of residue classes modulo $m$, $\tau(n)$ is the number of divisors of $n$, $\phi$ is Euler's totient function, $\zeta$ is the Riemann zeta-function, $\varepsilon$ denotes an arbitrary small positive number.

\section{\bf Preliminary lemmas}

We shall use the following lemmas.
\begin{lem} \label{Lemma_1} Suppose $a>0$, $b>0$, $T\geq 2$. Then we have
\begin{align*}
\frac{1}{2\pi i}\int_{b-iT}^{b+iT}\frac{a^{s}}{s}ds=\left\{\begin{array}{ll}
1+O(a^{b}\min(1,\frac{1}{T\log a})),&\mbox{ if $a>1,$}\\
O(a^{b}\min(1,\frac{1}{T|\log a|})),& \mbox{ if $0<a<1,$}\\
\frac{1}{2}+O(\frac{b}{T}),&\mbox{ if $a=1,$}
\end{array}\right.
\end{align*}
where the $O$-constants are absolute.
\end{lem}

\begin{proof} This is Lemma 6.5.1 of Pan and Pan \cite{PP}.
\end{proof}

\begin{lem} \label{Lemma_2} Suppose $\ell\geq 0$ is a fixed integer. For $\sigma>1$ we have the following estimate
\begin{align*}
\zeta^{(\ell)}(\sigma+it) &\ll \min\left(\frac {1}{(\sigma-1)^{1+\ell}}, \log^{1+\ell} (|t|+2)\right).
\end{align*}
\end{lem}

\begin{proof} The estimate for the case $\ell=0$ can be found in Chapter 7 of Pan and Pan \cite{PP}.
The estimate for the case $\ell\geq 1$ follows from the result of the case $\ell=0$ and Cauchy's
theorem.
\end{proof}

\begin{lem} \label{Lemma_3} Suppose $\ell\geq 0$ is a fixed integer. Then for $0\leq \sigma \leq 1$ we have
\begin{align*}
\zeta^{(\ell)}(\sigma+it)\ll (|t|+2)^{\frac{1-\sigma}{2}}\log^{1+\ell} (|t|+2).
\end{align*}
\end{lem}

\begin{proof} The estimate for $\ell=0$ can be found in Ivi\'c \cite{Iv}, see (1.66).
The estimate for the case $\ell\geq 1$ follows from the result of the case $\ell=0$ and Cauchy's theorem.
\end{proof}

\begin{lem} \label{Lemma_4} Suppose $\ell\geq 0$ is a fixed integer. Then for $1/2\leq \sigma \leq 1$ we have
\begin{align*}
\zeta^{(\ell)}(\sigma+it)\ll (|t|+2)^{\frac{1-\sigma}{3}}\log^{1+\ell} (|t|+2).
\end{align*}
\end{lem}

\begin{proof} The estimate for $\ell=0$ follows from the bounds
\begin{align*}
\zeta(1/2+it) \ll (|t|+2)^{1/6}, \ \
\zeta(1+it)  \ll  \log (|t|+2)
\end{align*}
and the Phragm\'{e}n-Lindel\"{o}f principle.
The estimate for $\ell\geq 1$ follows from the result of $\ell=0$ and Cauchy's theorem.
\end{proof}

\begin{lem} \label{Lemma_5} Suppose $V>10$ is a large parameter. Then
we have
\begin{align}
&\int_{-V}^{ V}| \zeta (u+iv)|^2dv  \ll V \ \ (0.6<u<2), \label{2.1}\\
&\int_{-V}^{ V}| \zeta (13/20+iv)|^8dv  \ll V^{1+\varepsilon}. \label{2.2}
\end{align}
\end{lem}

\begin{proof}
 The estimate (\ref{2.1}) holds for $u>1/2+\varepsilon$, see for example, (8.112) of Ivi\'c \cite{Iv}. The second estimate can be found in Chapter 8 of Ivi\'c \cite{Iv}.
\end{proof}

\begin{lem} \label{Lemma_6} If $\zeta(s)=\chi(s)\zeta(1-s)$, then the estimate
$$\chi(s)\ll (|t|+2)^{1/2-\sigma}$$
holds uniformly for $0\leq \sigma \leq 1$.
\end{lem}
\begin{proof} See (1.25) in Ivi\'{c} \cite{Iv}.
\end{proof}

\begin{lem} \label{Lemma_7} Let $\psi(t)=t-[t]-\frac{1}{2}$, $H\geq 3$. Then
$$
\psi(t)=-\sum_{1\leq |h|\leq H}\frac{e(ht)}{2\pi ih}+O\left(\min\left(1,\frac{1}{H\|t\|}\right)\right),
$$
where $\min\left(1,\frac{1}{H\|t\|}\right)=\sum\limits_{h=-\infty}^{\infty}a_{h}e(ht)$,
$a_{h}\ll \min\left(\frac{\log 2H}{H},\frac{1}{|h|},\frac{H}{h^{2}}\right)$.
\end{lem}

\begin{proof} The estimate can be found in Heath-Brown \cite{DR}.
\end{proof}

\begin{lem} \label{Lemma_8} Let $G(x):=\sum_{u^{2}v\leq x}u$, then we have the following estimate
$$
G(x)=\frac{1}{2}x\log x+\frac{3}{2}\gamma x-\frac{1}{2}x+O(x^{5/9}\log x),
$$
where $\gamma$ is the Euler-Mascheroni constant.
\end{lem}

\begin{proof}
Using Hyperbolic summation and Euler-Maclaurin summation formula we may write
\begin{align*}
G(x)&=\sum_{1\leq v\leq x^{1/3}}\sum_{1\leq u\leq \sqrt{x/v}}u
+\sum_{1\leq u\leq x^{1/3}}u\sum_{1\leq v \leq x/u^{2}}1-\sum_{v\leq x^{1/3}}\sum_{u\leq x^{1/3}}u\\
&=\sum_{1\leq v\leq x^{1/3}}\left(x/(2v)-\sqrt{x/v}\psi(\sqrt{x/v})+O(1)\right)
+\sum_{1\leq u\leq x^{1/3}}u [x/u^{2}]-\sum_{v\leq x^{1/3}}\sum_{u\leq x^{1/3}}u\\
&=\frac{1}{2}x\log x+\frac{3}{2}\gamma x-\frac{1}{2}x\\
&-\sum_{1\leq u\leq x^{1/3}}\left\{\sqrt{x/u}\psi(\sqrt{x/u})+u\psi(x/u^{2})\right\}+O(x^{1/3}),
\end{align*}
where $\psi(t)=t-[t]-1/2$.

Let $\Delta_{G}(x)$ be the error term of the asymptotic formula, it's easy to obtain
\begin{align*}
\Delta_{G}(x)=-\sum_{1\leq u\leq x^{1/3}}\left\{\sqrt{x/u}\psi(\sqrt{x/u})+u\psi(x/u^{2})\right\}+O(x^{1/3}).
\end{align*}
Then we will actually give
\begin{align*}
\Delta_{G}(x)\ll x^{(1+2\kappa+\lambda)/3(\kappa+1)},
\end{align*}
where $(\kappa,\lambda)$ is an exponent pair, and the value $\frac{5}{9}=0.555\ldots$ comes from choosing the exponent pair $(\kappa,\lambda)=(\frac{1}{2},\frac{1}{2})$. By Lemma \ref{Lemma_7}, it is seen that we have to estimate the exponential sum
\begin{align*}
S_{1}=\sum_{1\leq u\leq x^{1/3}}u e(hx/u^{2})
\end{align*}
and
\begin{align*}
S_{2}=\sum_{1\leq u\leq x^{1/3}}\sqrt{x/u} e(h \sqrt{x/u}).
\end{align*}
Now we only estimate $S_{1}$, the proof of $S_{2}$ is similar. We start by proving
\begin{align*}
\sum_{P< u\leq 2P}u e(hx/u^{2})\ll h^{\kappa}x^{\kappa}P^{\lambda+1-3\kappa},
\end{align*}
where $P=x^{1/3}/2^{k+1}$, $k\geq 0$.
Let $f(u)=hx/u^{2}$, then $|f^{'}|\asymp hx/P^{3}$, we use the theory of exponent pairs to obtain
\begin{align*}
\sum_{P< u\leq 2P}e(hx/u^{2})\ll (hx/P^{3})^{\kappa}P^{\lambda}=h^{\kappa}x^{\kappa}P^{\lambda-3\kappa}.
\end{align*}
In this case with partial summation we have
\begin{equation}\label{exponential sum}
\sum_{P< u\leq 2P}u e(hx/u^{2})\ll h^{\kappa}x^{\kappa}P^{\lambda+1-3\kappa}.
\end{equation}
By Lemma \ref{Lemma_7} and (\ref{exponential sum}), we get
\begin{align*}
\sum_{P< u\leq 2P}u\psi(x/u^{2})
&\ll -\sum_{1\leq |h|\leq H}\frac{1}{2\pi ih}\sum_{P< u\leq 2P}u e(hx/u^{2})+\sum_{P< u\leq 2P}u \min\left(1,1/(H\|x/u^{2}\|)\right)\\
&\ll \log H P^{2}H^{-1}+ h^{\kappa}x^{\kappa}P^{\lambda+1-3\kappa} \ll x^{(1+2\kappa+\lambda)/3(\kappa+1)},
\end{align*}
with $H=P^{(1-\lambda+3\kappa)/(\kappa+1)}x^{-\kappa/(\kappa+1)}$.

Similarly, we have
\begin{align*}
\sum_{P< u\leq 2P} \sqrt{x/u}\psi(\sqrt{x/u}) \ll x^{(1+2\kappa+\lambda)/3(\kappa+1)}.
\end{align*}

This completes the proof of Lemma \ref{Lemma_8}.
\end{proof}

\section{\bf On the Dirichlet series of $s(m_{1},m_{2})$}
In this section we shall study the Dirichlet series
\begin{align}
F(s;k):=\sum_{\substack{m_{1},m_{2}\geq 1 \\ m_{1}m_{2}\in N_{k}}}\frac{s(m_{1},m_{2})}{m_{1}^{s}m_{2}^{s}} \ \ (\Re s>1).
\end{align}
For simplicity, we write
\begin{align}\label{f(n,s)}
f(n^{k};s):=\sum_{n^{k}=m_{1}m_{2}}\frac{s(m_{1},m_{2})}{m_{1}^{s}m_{2}^{s}},
\end{align}
then it follows that
\begin{align}\label{f(n,s)-Dirichlet series}
\sum_{n=1}^{\infty}f(n^{k};s)=\sum_{\substack{m_{1},m_{2}\geq 1 \\ m_{1}m_{2}\in N_{k}}}\frac{s(m_{1},m_{2})}{m_{1}^{s}m_{2}^{s}}.
\end{align}
Noting that $f(n^{k};s)$ is multiplicative, we have
\begin{align}\label{f(n,s)-Euler product}
\sum_{n=1}^{\infty}f(n^{k};s)=\prod_{p}\sum_{\alpha=0}^{\infty}f(p^{\alpha k};s),
\end{align}
where
\begin{align}\label{f(p,s)}
f(p^{\alpha k};s)=\sum_{p^{\alpha k}=m_{1}m_{2}}\frac{s(m_{1},m_{2})}{m_{1}^{s}m_{2}^{s}}.
\end{align}
In particular, when $k$ takes 2 and 3, we have the following Proposition 3.1 and Proposition 3.2, respectively.
We first consider the case when $k=2$.
\begin{prop} \label{Proposition_1} Suppose that $s=\sigma+it$, $\sigma>1$. Then we have
\begin{align}
F(s;2)=\zeta(2s-1)H(s;2),
\end{align}
where $H(s;2)$ can be written as a Dirichlet series, which is absolutely convergent when $\sigma> 1/2$.
\end{prop}
\begin{proof}
Indeed, if $k=2$, then
\begin{align}\label{f(p,s)-2}
f(p^{2\alpha};s)=\sum_{p^{2\alpha}=m_{1}m_{2}}\frac{s(m_{1},m_{2})}{m_{1}^{s}m_{2}^{s}}.
\end{align}
To estimate (\ref{f(p,s)-2}), we shall deal with the sum $\sum_{p^{2\alpha}=m_{1}m_{2}}s(m_{1},m_{2})$. Let $m_{1}=p^{u}$, $m_{2}=p^{v}$, then
\begin{align}\label{s(p,s)-2}
\sum_{p^{2\alpha}=m_{1}m_{2}}s(m_{1},m_{2})
=\sum_{u+v=2\alpha}s(p^{u},p^{v}),
\end{align}
where $0\leq u\leq 2\alpha$, $0\leq v\leq 2\alpha$.
From (\ref{s(m,n)}), we may write
\begin{align*}
s\left(p^{u},p^{v}\right)=\sum_{d_{1}|p^{u}, d_{2}|p^{v}}\gcd(d_{1},d_{2}).
\end{align*}
Let $d_{1}=p^{u_{1}}$, $d_{2}=p^{u_{2}}$ and $0\leq u_{1}\leq u$, $0\leq u_{2}\leq v$. Then for instance we obtain
\begin{align}\label{s(p^{u},p^{v})}
s\left(p^{u},p^{v}\right)&=\sum_{u_{1}\leq u, u_{2}\leq v}\gcd(p^{u_{1}},p^{u_{2}})
=\sum_{u_{1}\leq u, u_{2}\leq v}p^{\min (u_{1},u_{2})}\nonumber\\
&=p^{\min (u,v)}+O(p^{\min (u,v)-1}).
\end{align}
Using the estimate of (\ref{s(p^{u},p^{v})}), we have
\begin{align}\label{s(p^{u},p^{v})-sum}
\sum_{u+v=2\alpha}s(p^{u},p^{v})
=s(p^{\alpha},p^{\alpha})+2\sum_{u<\alpha}s(p^{u},p^{2\alpha-u})
=p^{\alpha}+O(p^{\alpha-1}).
\end{align}
By (\ref{f(p,s)-2}), (\ref{s(p,s)-2}) and (\ref{s(p^{u},p^{v})-sum}), we have
\begin{align*}
f(p^{2\alpha};s)=\frac{p^{\alpha}+O(p^{\alpha-1})}{p^{2\alpha s}}
=\frac{1}{p^{\alpha(2s-1)}}+O\bigg(\frac{1}{p^{\alpha(2s-1)+1}}\bigg).
\end{align*}
Hence
\begin{align}\label{f(p,s)-Dirichlet series}
\sum_{\alpha=0}^{\infty}f(p^{2\alpha};s)
&=1+\frac{1}{p^{2s-1}}+\frac{1}{p^{2(2s-1)}}+\frac{1}{p^{3(2s-1)}}+\cdots+O\bigg(\frac{1}{p^{2s}} \bigg)\nonumber\\
&=1+\frac{p^{-(2s-1)}}{1-p^{-(2s-1)}}+O\bigg(\frac{1}{p^{2s}} \bigg).
\end{align}
It is easy to see that
\begin{align}
\bigg(1-p^{-(2s-1)}\bigg)\bigg(1+\frac{p^{-(2s-1)}}{1-p^{-(2s-1)}}+O\bigg(\frac{1}{p^{2s}} \bigg)\bigg)
=1+O\bigg(\frac{1}{p^{2s}} \bigg).
\end{align}
We write
\begin{align}
\sum_{\alpha=0}^{\infty}f(p^{2\alpha};s)=\bigg(1-p^{-(2s-1)}\bigg)^{-1}\times h(p,s;2),
\end{align}
where
\begin{align}\label{h(p,s;2)}
h(p,s;2)=\bigg(1-p^{-(2s-1)}\bigg)\times \sum_{\alpha=0}^{\infty}f(p^{2\alpha};s)
=1+O\bigg(\frac{1}{p^{2s}} \bigg).
\end{align}
Now we see from (\ref{f(p,s)-Dirichlet series})--(\ref{h(p,s;2)}) that Proposition 3.1 holds and $H(s;2)$ can be written as a Dirichlet series which is absolutely convergent with $\sigma> 1/2$.
\end{proof}

If $k=3$, then we have the following proposition.
\begin{prop} \label{Proposition_2} Suppose that $s=\sigma+it$, $\sigma>1$. Then we have
\begin{align}
F(s;3)=\zeta^{4}(3s-1)\zeta(6s-3)H(s;3),
\end{align}
where $H(s;3)$ can be written as a Dirichlet series, which is absolutely convergent when $\sigma> 1/2$.
\end{prop}
\begin{proof}
When $k=3$, we have
\begin{align}\label{f(p,s)-3}
f(p^{3\alpha};s)=\sum_{p^{3\alpha}=m_{1}m_{2}}\frac{s(m_{1},m_{2})}{m_{1}^{s}m_{2}^{s}}.
\end{align}
Similar to (\ref{f(p,s)-2}), to estimate (\ref{f(p,s)-3}), we shall deal with the following sum
\begin{align}\label{s(p,s)-3}
\sum_{p^{3\alpha}=m_{1}m_{2}}s(m_{1},m_{2})
=\sum_{u+v=3\alpha}s(p^{u},p^{v}),
\end{align}
where $0\leq u\leq 3\alpha$, $0\leq v\leq 3\alpha$. Let's consider the cases where $\alpha$ is odd and even respectively.

\noindent
Case 1: suppose $\alpha$ is odd. According to the estimate of (\ref{s(p^{u},p^{v})}), we have
\begin{align}\label{s(p^{u},p^{v})-odd}
s(p^{\frac{3\alpha-1}{2}},p^{\frac{3\alpha+1}{2}})
=\sum_{\substack{u_{1}\leq (3\alpha-1)/2 \\ u_{2}\leq (3\alpha+1)/2}}\gcd(p^{u_{1}},p^{u_{2}})
=2p^{\frac{3\alpha-1}{2}}+O(p^{\frac{3\alpha-1}{2}-1}),
\end{align}
which gives
\begin{align}\label{s(p^{u},p^{v})-odd 3 sum}
\sum_{u+v=3\alpha}s(p^{u},p^{v})
&=2s(p^{\frac{3\alpha-1}{2}},p^{\frac{3\alpha+1}{2}})+2\sum_{u<3\alpha/2}s(p^{u},p^{3\alpha-u})\nonumber\\
&=4p^{\frac{3\alpha-1}{2}}+O(p^{\frac{3\alpha-1}{2}-1}).
\end{align}
By (\ref{f(p,s)-3}), (\ref{s(p,s)-3}) and (\ref{s(p^{u},p^{v})-odd 3 sum}), we get
\begin{align*}
f(p^{3\alpha};s)=\frac{4p^{\frac{3\alpha-1}{2}}+O(p^{\frac{3\alpha-1}{2}-1})}{p^{3\alpha s}}
=\frac{4p^{\frac{\alpha-1}{2}}}{p^{\alpha(3s-1)}}+O\bigg(\frac{1}{p^{3\alpha s-3\alpha/2+3/2}}\bigg).
\end{align*}
Hence
\begin{align}\label{f(p,s) odd 3-Dirichlet series}
\sum_{\substack{\alpha=1 \\ \alpha \ {\rm is} \ {\rm odd}}}^{\infty}f(p^{3\alpha};s)
&=\frac{4}{p^{3s-1}}+\frac{4p}{p^{3(3s-1)}}+\frac{4p^{2}}{p^{5(3s-1)}}+\cdots+O\bigg(\frac{1}{p^{3s}} \bigg)\nonumber\\
&=\frac{4p^{-(3s-1)}}{1-p^{-(6s-3)}}+O\bigg(\frac{1}{p^{3s}} \bigg).
\end{align}
\noindent
Case 2: suppose $\alpha$ is even. Using the estimate of (\ref{s(p^{u},p^{v})}), we have
\begin{align}\label{s(p^{u},p^{v})-even 3 sum}
\sum_{u+v=3\alpha}s(p^{u},p^{v})
&=s(p^{\frac{3\alpha}{2}},p^{\frac{3\alpha}{2}})+2\sum_{u<3\alpha/2}s(p^{u},p^{3\alpha-u})\nonumber\\
&=p^{\frac{3\alpha}{2}}+O(p^{\frac{3\alpha}{2}-1}).
\end{align}
By (\ref{f(p,s)-3}), (\ref{s(p,s)-3}) and (\ref{s(p^{u},p^{v})-even 3 sum}), we find
\begin{align*}
f(p^{3\alpha};s)=\frac{p^{\frac{3\alpha}{2}}+O(p^{\frac{3\alpha}{2}-1})}{p^{3\alpha s}}
=\frac{1}{p^{3\alpha s-3\alpha/2}}+O\bigg(\frac{1}{p^{3\alpha s-3\alpha/2+1}}\bigg).
\end{align*}
Hence
\begin{align}\label{f(p,s) even 3-Dirichlet series}
\sum_{\substack{\alpha=0 \\ \alpha \ {\rm is} \ {\rm even}}}^{\infty}f(p^{3\alpha};s)
&=1+\frac{1}{p^{6s-3}}+\frac{1}{p^{2(6s-3)}}+\frac{1}{p^{3(6s-3)}}+\cdots+O\bigg(\frac{1}{p^{6s-2}} \bigg)\nonumber\\
&=1+\frac{p^{-(6s-3)}}{1-p^{-(6s-3)}}+O\bigg(\frac{1}{p^{6s-2}} \bigg).
\end{align}
Combining (\ref{f(p,s) odd 3-Dirichlet series}) and (\ref{f(p,s) even 3-Dirichlet series}), we get
\begin{align}\label{f(p,s) 3-Dirichlet series}
\sum_{\alpha=0}^{\infty}f(p^{3\alpha};s)
=1+\frac{4p^{-(3s-1)}}{1-p^{-(6s-3)}}+\frac{p^{-(6s-3)}}{1-p^{-(6s-3)}}+O\bigg(\frac{1}{p^{3s}}+\frac{1}{p^{6s-2}} \bigg).
\end{align}
For simplicity, we write
$$w_{1}:=6s-3,\ \  w_{2}:=3s-1.$$
So, from (\ref{f(p,s) 3-Dirichlet series}), we can write
\begin{align}\label{f(p,s) 3-asymptotic}
\sum_{\alpha=0}^{\infty}f(p^{3\alpha};s)
=1+\frac{4p^{-w_{2}}}{1-p^{-w_{1}}}+\frac{p^{-w_{1}}}{1-p^{-w_{1}}}+B(p,s;3),
\end{align}
where
\begin{align*}
B(p,s;3)=O\bigg(\frac{1}{p^{3s}}+\frac{1}{p^{6s-2}} \bigg).
\end{align*}
After some calculation, we obtain
\begin{align}
&\ \ \ (1-p^{-w_{1}})(1-p^{-w_{2}})^{4}\bigg(1+\frac{p^{-w_{1}}}{1-p^{-w_{1}}}+\frac{4p^{-w_{2}}}{1-p^{-w_{1}}}\bigg)\nonumber\\
&=(1-p^{-w_{1}})(1-p^{-w_{2}})^{4}
\bigg(1+p^{-w_{1}}+\sum_{\ell=2}^{\infty}p^{-\ell w_{1}}+4p^{-w_{2}}+\sum_{m=1}^{\infty}p^{-(mw_{1}+w_{2})}\bigg)\nonumber\\
&=1+O(p^{-2\Re w_{2}}).
\end{align}
We write
\begin{align}\label{f(p,s;3)}
\sum_{\alpha=0}^{\infty}f(p^{3\alpha};s)=(1-p^{-w_{1}})^{-1}(1-p^{-w_{2}})^{-4}\times h(p,s;3),
\end{align}
where
\begin{align}\label{h(p,s;3)}
&\ \ \ h(p,s;3)=(1-p^{-w_{1}})(1-p^{-w_{2}})^{4}\times \sum_{\alpha=0}^{\infty}f(p^{3\alpha};s)\nonumber\\
&=(1-p^{-w_{1}})(1-p^{-w_{2}})^{4}\times
\bigg(1+\frac{p^{-w_{1}}}{1-p^{-w_{1}}}+\frac{4p^{-w_{2}}}{1-p^{-w_{1}}}+B(p,s;3)\bigg)\nonumber\\
&=1+O\bigg(\frac{1}{p^{3s}}+\frac{1}{p^{6s-2}} \bigg).
\end{align}
Now we see from (\ref{f(p,s;3)}) and (\ref{h(p,s;3)}) that Proposition 3.1 holds and $H(s;3)$ can be written as a Dirichlet series which is absolutely convergent with $\sigma> 1/2$.
\end{proof}

\section{\bf Proof of Theorem \ref{Th_1}}

In this section, we shall prove Theorem \ref{Th_1}.

\subsection{\bf Application of Perron's  formula}

In this subsection, We shall prove a variant of the Perron formula, which is applicable in the proof of Theorem 1.1.

Suppose $B_{k}(\sigma)$ is a function such that
$$\sum_{\substack{m_{1},m_{2}\geq 1 \\ m_{1}m_{2}\in N_{k}}}\frac{|s(m_{1},m_{2})|}{(m_{1}m_{2})^{\sigma}}\ll B_{k}(\sigma),\ \ \ \ \sigma>1.$$
Then we have the following proposition.

\begin{prop} \label{Proposition_1} Suppose $x$ is a large parameter, $10<T\leq x$, $1<b<2$. Then we have the formula
$$
\sum_{\substack{m_{1}m_{2}\leq x \\ m_{1}m_{2}\in N_{k}}}s(m_{1},m_{2})= \frac{1}{2\pi i}\int_{b-iT}^{b+iT}F(s;k)\frac{x^{s}}{s}ds+O\left(\frac{x^{1+\varepsilon}}{T}+x^{5/9+\varepsilon}\right).
$$
\end{prop}

\begin{proof}
By Lemma \ref{Lemma_1}, we have
\begin{align*}
\frac{1}{2\pi i}\int_{b-iT}^{b+iT}F(s;k)\frac{x^{s}}{s}ds
&=\sum_{\substack{m_{1},m_{2}\geq 1 \\ m_{1}m_{2}\in N_{k}}}s(m_{1},m_{2})\times \frac{1}{2\pi i}\int_{b-iT}^{b+iT}\left(\frac{x}{m_{1}m_{2}}\right)^{s}\frac{ds}{s}\\
&=\sum_{\substack{m_{1}m_{2}\leq x \\ m_{1}m_{2}\in N_{k}}}s(m_{1},m_{2})+O(R),
\end{align*}
where
\begin{align*}
R=\sum_{\substack{m_{1},m_{2}\geq 1 \\ m_{1}m_{2}\in N_{k}}}\left|s(m_{1},m_{2})\right|\left(\frac{x}{m_{1}m_{2}}\right)^{b}\min\left(1,\frac{1}{T|\log \frac{x}{m_{1}m_{2}}|}\right).
\end{align*}

We divide the above sum into three parts:
\begin{align*}
R=\sum_{\substack{m_{1}m_{2}\leq x/2 \\ m_{1}m_{2}\in N_{k}}}
 +\sum_{\substack{x/2<m_{1}m_{2}\leq 2x \\ m_{1}m_{2}\in N_{k}}}
 +\sum_{\substack{m_{1}m_{2}> 2x \\ m_{1}m_{2}\in N_{k}}}.
\end{align*}

It is easily seen that
\begin{align*}
 \sum_{\substack{m_{1}m_{2}\leq x/2 \\ m_{1}m_{2}\in N_{k}}}
 +\sum_{\substack{m_{1}m_{2}> 2x \\ m_{1}m_{2}\in N_{k}}}
 \ll \frac{x^{b}}{T}\sum_{\substack{m_{1},m_{2}\geq 1 \\ m_{1}m_{2}\in N_{k}}}\frac{|s(m_{1},m_{2})|}{(m_{1}m_{2})^{b}}
 \ll\frac{x^{b}B_{k}(b)}{T}.
\end{align*}

It only needs to estimate the sum $\sum\limits_{\substack{x/2<m_{1}m_{2}\leq 2x \\ m_{1}m_{2}\in N_{k}}}$.
From (\ref{s(m,n)}) we have
\begin{align*}
s(m_{1},m_{2})= \sum_{d\mid m_{1},\, e\mid m_{2}} \gcd(d,e)\leq \gcd(m_{1},m_{2})\tau(m_{1})\tau(m_{2}).
\end{align*}
Suppose $\gcd(m_{1},m_{2})=\ell$, $m_{1}=\ell u$, $m_{2}=\ell v$ and $(u,v)=1$, noting that $\tau(m_{1}m_{2})\leq \tau(m_{1})\tau(m_{2})$ and $\tau(m_{1})\ll m_{1}^{\varepsilon}$, then
we can get
\begin{align*}
\sum_{\substack{x/2<m_{1}m_{2}\leq 2x \\ m_{1}m_{2}\in N_{k}}}
&\leq \sum_{\substack{x/2<\ell^{2}(uv)\leq 2x \\ (u,v)=1}}
\ell \tau(\ell u)\tau(\ell v) \min\left(1,\frac{1}{T|\log \frac{x}{\ell^{2}(uv)}|}\right)\\
&\leq \sum_{\substack{x/2<\ell^{2}(uv)\leq 2x \\ (u,v)=1}}
\ell \tau^{2}(\ell)\tau(u)\tau(v) \min\left(1,\frac{1}{T|\log \frac{x}{\ell^{2}(uv)}|}\right)\\
&\ll x^{\varepsilon}\sum_{\substack{x/2<\ell^{2}(uv)\leq 2x \\ (u,v)=1}}
\ell \ \   \min\left(1,\frac{1}{T|\log \frac{x}{\ell^{2}(uv)}|}\right)\\
&=x^{\varepsilon} \left(\sum\nolimits_{1}+\sum\nolimits_{2}+\sum\nolimits_{3}\right).
\end{align*}
where
\begin{align*}
&\sum\nolimits_{1}:=
\sum_{\substack{x/2<\ell^{2}(uv)\leq xe^{-1/T} \\ (u,v)=1}} \ell  \ \  \min\left(1,\frac{1}{T|\log \frac{x}{\ell^{2}(uv)}|}\right),\\
&\sum\nolimits_{2}:=
\sum_{\substack{xe^{-1/T}<\ell^{2}(uv)\leq xe^{1/T} \\ (u,v)=1}} \ell  \ \  \min\left(1,\frac{1}{T|\log \frac{x}{\ell^{2}(uv)}|}\right),\\
&\sum\nolimits_{3}:=
\sum_{\substack{xe^{1/T}<\ell^{2}(uv)\leq 2x \\ (u,v)=1}} \ell \ \   \min\left(1,\frac{1}{T|\log \frac{x}{\ell^{2}(uv)}|}\right).
\end{align*}
Let $w=uv$ and $\tau(w)$ denote the number of divisors $w$. Then there is $\tau(w)\ll w^{\varepsilon}$. For $\sum_{2}$ we have
\begin{align*}
\sum\nolimits_{2}=\sum_{xe^{-1/T}< \ell^{2}w \leq xe^{1/T}} \tau(w)\ell
\ll x^{\varepsilon}\sum_{xe^{-1/T}< \ell^{2}w \leq xe^{1/T}} \ell.
\end{align*}
Firstly we consider $\sum\limits_{xe^{-1/T}< \ell^{2}w \leq xe^{1/T}} \ell$.
By Lemma \ref{Lemma_8} we have
\begin{align*}
\sum_{\ell^{2}w\leq x}\ell=\frac{1}{2}x\log x+(\frac{3}{2}\gamma-\frac{1}{2})x+O(x^{5/9}\log x).
\end{align*}
Let $C(x)$ denote the main term of $\sum\limits_{\ell^{2}w\leq x}\ell$, and using Lagrange mean value theorem, we get
\begin{align*}
C(xe^{1/T})-C(xe^{-1/T})\ll x\log x(e^{1/T}-e^{-1/T})\ll \frac{x\log x}{T}.
\end{align*}
Then we can deduce
\begin{align*}
\sum_{xe^{-1/T}< \ell^{2}w \leq xe^{1/T}} \ell
\ll C(xe^{1/T})-C(xe^{-1/T})+x^{5/9}\log x \ll \frac{x\log x}{T}+x^{5/9}\log x.
\end{align*}
Thus
\begin{align*}
\sum\nolimits_{2}\ll \frac{x^{1+\varepsilon}}{T}+x^{5/9+\varepsilon}.
\end{align*}
For $\sum_{3}$, let $\ell^{2}w=[x]+r$, we have
\begin{align*}
\frac{1}{|\log (x/\ell^{2}w)|}=\frac{1}{|\log x/([x]+r|)}=\frac{1}{|\log x -\log (x-\{x\}+r)|} \ll \frac{x}{r}.
\end{align*}
Hence by Lemma \ref{Lemma_8} and Lagrange mean value theorem, we can deduce
\begin{equation}
\begin{split}
\sum\nolimits_{3}&= \frac{1}{T}\sum_{xe^{1/T}< \ell^{2}w \leq 2x} \tau(w)\ell \ \   \frac{1}{|\log (x/\ell^{2}w)|} \ll \frac{x^{\varepsilon}}{T}\sum_{xe^{1/T}< \ell^{2}w \leq 2x} \ell \ \   \frac{x}{\ell^{2}w-[x]} \\
&= \frac{x^{\varepsilon}}{T}\sum_{xe^{1/T}< \ell^{2}w \leq 2x} \ell \ \   \frac{x}{\ell^{2}w(1-[x]/\ell^{2}w)} \ll \frac{x^{\varepsilon}}{T}\sum_{xe^{1/T}< \ell^{2}w \leq 2x} \ell \\
&\ll \frac{x^{\varepsilon}}{T}\left(x\log x+x^{5/9}\log x\right)\ll \frac{x^{1+\varepsilon}}{T}. \nonumber\\
\end{split}
\end{equation}
Similarly, we can deduce
\begin{align*}
\sum\nolimits_{1} \ll \frac{x^{1+\varepsilon}}{T}.
\end{align*}
Let $b=1+1/\log x$, we immediately complete the proof of Proposition $4.1$. In particular, when $k=2$, take $B_{2}(\sigma)=1/(\sigma-1)$, and when $k=3$, take $B_{3}(\sigma)=(\sigma-1)^{-5}$.
\end{proof}

Then by Proposition $4.1$, we have
\begin{align} \label{sum-4.1}
\sum_{\substack{m_{1}m_{2}\leq x \\ m_{1}m_{2}\in N_{2}}}s(m_{1},m_{2})
&=\frac{1}{2\pi i}\int_{b-iT}^{b+iT}F(s;2)\frac{x^{s}}{s}ds+O\left(\frac{x^{1+\varepsilon}}{T}+x^{5/9+\varepsilon}\right)\nonumber\\
&:=I(x,T;2)+O\left(\frac{x^{1+\varepsilon}}{T}+x^{5/9+\varepsilon}\right),
\end{align}
where
\begin{equation} \label{I(x,T)-1}
I(x,T;2)=\frac{1}{2\pi i}\int_{b-iT}^{b+iT}\zeta(2s-1)H(s;2)\frac{x^{s}}{s}ds.
\end{equation}

\subsection{\bf Evaluation of the integral $I(x,T;2)$}

Now we estimate $I(x,T;2)$. Consider the rectangle domain formed by the four points $s = b\pm iT$, $s = \frac{11}{20}\pm iT$.
In this domain the integrand
\begin{equation*}
g(s;2):= \zeta(2s-1)H(s;2)\frac{x^{s}}{s}
\end{equation*}
has a pole of order 1, namely $s = 1$. By the residue theorem we have
\begin{equation} \label{I(x,T;2)}
I(x,T;2)=J(x,T;2)+H_{1}(x,T;2)+H_{2}(x,T;2)-H_{3}(x,T;2),
\end{equation}
where
\begin{equation}
\begin{split}
J(x,T;2)&:=\Res_{s=1}\zeta(2s-1)H(s;2)\frac{x^{s}}{s},\\
H_{1}(x,T;2)&:=\frac{1}{2\pi i}\int_{11/20+iT}^{b+iT}\zeta(2s-1)H(s;2)\frac{x^{s}}{s}ds,\\
H_{2}(x,T;2)&:=\frac{1}{2\pi i}\int_{11/20-iT}^{11/20+iT}\zeta(2s-1)H(s;2)\frac{x^{s}}{s}ds,\\
H_{3}(x,T;2)&:=\frac{1}{2\pi i}\int_{11/20-iT}^{b-iT}\zeta(2s-1)H(s;2)\frac{x^{s}}{s}ds.\nonumber\\
\end{split}
\end{equation}
Obviously, since $s=1$ is the pole of $g(s;2)$ of degree 1, we have
\begin{equation} \label{residue 1}
\textrm{Res}_{s=1}\zeta(2s-1)H(s;2)\frac{x^{s}}{s}=c_{2}x,
\end{equation}
where $c_{2}$ is a computable constant.

Now we estimate $H_{1}(x,T;2)$. In this case by Lemma \ref{Lemma_3} we have (noting that $|t|\leq T$), uniformly for $\frac{11}{20}\leq \sigma\leq b=1+\frac{1}{\log x}$,
\begin{equation}
g(s;2)=g(\sigma+iT;2)\ll \left|\zeta(2\sigma-1+2iT)\right|\frac{x^{\sigma}}{T}
\ll\frac{x^{\sigma}\log x}{T} T^{\max(1-\sigma,0)}. \nonumber\\
\end{equation}
So we get
\begin{align} \label{H_{1}(x,T;2)}
H_{1}(x,T;2)
&\ll  \frac{\log x}{T}\bigg(\int_{\frac{11}{20}}^{1}x^{\sigma}T^{(1-\sigma)}d\sigma
+\int_{1}^{b}x^{\sigma}d\sigma\bigg)\nonumber\\
&\ll \frac{x\log x}{T}+ \frac{x^{11/20}\log x}{T^{11/20}}\ll \frac{x\log x}{T}.
\end{align}

Similarly, we have
\begin{equation} \label{H_{3}(x,T;2)}
H_{3}(x,T;2)\ll \frac{x\log x}{T}.
\end{equation}

Next we estimate $H_{2}(x,T;2)$.  In this case by Lemma \ref{Lemma_6} we get, with $|t|\leq T$,
\begin{equation}
\begin{split}
H_{2}(x,T;2)
&\ll \int^{T}_{-T}\frac{\left|\zeta(1/10+2it)\right|}{|t|+1}x^{11/20}dt\\
&\ll x^{11/20} \left(1+\int_{1}^{T}\frac{\left|\zeta(1/10+2it)\right|}{t}dt\right) \\
&\ll x^{11/20}+x^{11/20} \int_{1}^{T}\frac{t^{2/5}\left|\zeta(9/10-2it)\right|}{t}dt \\
&\ll x^{11/20}+x^{11/20} \int_{1}^{T}\frac{\left|\zeta(9/10-2it)\right|}{t^{3/5}}dt. \nonumber\\
\end{split}
\end{equation}
From the estimate (\ref{2.1}) and Cauchy's inequality we have
\begin{equation*}
\int_{0}^{T}|\zeta(9/10+iy)|dy\ll T.
\end{equation*}
Let $L_1(v):=\int_0^v|\zeta(9/10+iy)|dy$. Using partial summation we obtain
\begin{align} \label{4.7}
\int_{1}^V \frac{|\zeta(9/10+iv)|}{v^{3/5}}dv
&=\int_{1}^V \frac{dL_1(v)}{v^{3/5}}\ll \frac{L_1(V)}{V^{3/5}}+\int_{1}^V \frac{L_1(v)}{v^{8/5}}\nonumber\\
&\ll V^{2/5}+\int_{1}^V v^{-3/5}dv\ll V^{2/5}.
\end{align}
Hence
\begin{align} \label{H_{2}(x,T;2)}
H_{2}(x,T;2)\ll x^{11/20}T^{2/5}.
\end{align}

Thus from (\ref{I(x,T;2)}), (\ref{residue 1}), (\ref{H_{1}(x,T;2)}), (\ref{H_{3}(x,T;2)}) and (\ref{H_{2}(x,T;2)}) we obtain
\begin{equation} \label{I(x,T)-2}
I(x,T;2)=c_{2}x+O(x^{19/28+\varepsilon})
\end{equation}
by choosing $T=x^{9/28}$, where $c_{2}$ is an absolute constant.

\section{\bf Proof of Theorem \ref{Th_2}}

In this section, we shall prove Theorem \ref{Th_2}. By Proposition $4.1$, we have
\begin{align} \label{sum-4.1 f(s,3)}
\sum_{\substack{m_{1}m_{2}\leq x \\ m_{1}m_{2}\in N_{3}}}s(m_{1},m_{2})
&=\frac{1}{2\pi i}\int_{b-iT}^{b+iT}F(s;3)\frac{x^{s}}{s}ds+O\left(\frac{x^{1+\varepsilon}}{T}+x^{5/9+\varepsilon}\right)\nonumber\\
&:=I(x,T;3)+O\left(\frac{x^{1+\varepsilon}}{T}+x^{5/9+\varepsilon}\right),
\end{align}
where
\begin{equation} \label{I(x,T;3)}
I(x,T;3)=\frac{1}{2\pi i}\int_{b-iT}^{b+iT}\zeta^{4}(3s-1)\zeta(6s-3)H(s;3)\frac{x^{s}}{s}ds.
\end{equation}
Next, we mainly estimate $I(x,T;3)$. Consider the rectangle domain formed by the four points $s = b\pm iT$, $s = \frac{11}{20}\pm iT$.
In this domain the integrand
\begin{equation*}
g(s;3):=\zeta^{4}(3s-1)\zeta(6s-3)H(s;3)\frac{x^{s}}{s}
\end{equation*}
has a pole of order 5, namely $s = 2/3$. By the residue theorem we have
\begin{equation} \label{I(x,T;3)}
I(x,T;3)=J(x,T;3)+H_{1}(x,T;3)+H_{2}(x,T;3)-H_{3}(x,T;3),
\end{equation}
where
\begin{equation}
\begin{split}
J(x,T;3)&:=\Res_{s=2/3}\zeta^{4}(3s-1)\zeta(6s-3)H(s;3)\frac{x^{s}}{s},\\
H_{1}(x,T;3)&:=\frac{1}{2\pi i}\int_{11/20+iT}^{b+iT}\zeta^{4}(3s-1)\zeta(6s-3)H(s;3)\frac{x^{s}}{s}ds,\\
H_{2}(x,T;3)&:=\frac{1}{2\pi i}\int_{11/20-iT}^{11/20+iT}\zeta^{4}(3s-1)\zeta(6s-3)H(s;3)\frac{x^{s}}{s}ds,\\
H_{3}(x,T;3)&:=\frac{1}{2\pi i}\int_{11/20-iT}^{b-iT}\zeta^{4}(3s-1)\zeta(6s-3)H(s;3)\frac{x^{s}}{s}ds.\nonumber\\
\end{split}
\end{equation}

Obviously, since $s=2/3$ is the pole of $g(s;3)$ of degree 5, we have
\begin{equation} \label{residue 2/3}
{\rm Res}_{s=\frac{2}{3}} \zeta^{4}(3s-1)\zeta(6s-3)H(s;3)\frac{x^{s}}{s}=x^{2/3}\sum_{j=0}^{4}a_{j}\log^{j}x,
\end{equation}
where $a_{j} \ (0\leq j\leq 4)$ are computable constants.

Now we estimate $H_{1}(x,T;3)$. In this case by Lemma \ref{Lemma_3} and Lemma \ref{Lemma_4}, we have (noting that $|t|\leq T$), uniformly for $\frac{11}{20}\leq \sigma\leq b=1+\frac{1}{\log x}$,
\begin{align*}
g(s;3)=g(\sigma+iT;3)&\ll |\zeta(3\sigma-1+3iT)|^{4}|\zeta(6\sigma-3+6iT)|\frac{x^{\sigma}}{T} \nonumber\\
&\ll\frac{x^{\sigma}\log^{5} x}{T} T^{\frac{7}{3}\max(2-3\sigma, 0)}.
\end{align*}
So we get
\begin{align} \label{H_{1}(x,T;3)}
H_{1}(x,T;3)
&\ll  \frac{\log^{5} x}{T}\bigg(\int_{\frac{11}{20}}^{2/3}x^{\sigma}T^{\frac{7}{3}(2-3\sigma)}d\sigma
+\int_{2/3}^{b}x^{\sigma}d\sigma\bigg)\nonumber\\
&\ll \frac{x\log^{5} x}{T}+ \frac{x^{11/20}\log^{5} x}{T^{11/60}}\ll \frac{x\log^{5} x}{T}.
\end{align}

Similarly, we have
\begin{equation} \label{H_{3}(x,T;3)}
H_{3}(x,T;3)\ll \frac{x\log^{5} x}{T}.
\end{equation}

Next we estimate $H_{2}(x,T;3)$.  In this case by Lemmas \ref{Lemma_6} we get, with $|t|\leq T$,
\begin{equation}
\begin{split}
H_{2}(x,T;3)
&\ll \int^{T}_{-T}\frac{|\zeta(13/20+3it)|^{4}|\zeta(3/10+6it)|}{|t|+1}x^{11/20}dt\\
&\ll x^{11/20} \left(1+\int_{1}^{T}\frac{|\zeta(13/20+3it)|^{4}|\zeta(3/10+6it)|}{t}dt\right) \\
&\ll x^{11/20}+x^{11/20} \int_{1}^{T}\frac{t^{1/5}|\zeta(13/20+3it)|^{4}|\zeta(7/10-6it)|}{t}dt \\
&\ll x^{11/20}+x^{11/20} \int_{1}^{T}\frac{|\zeta(13/20+3it)|^{4}|\zeta(7/10-6it)|}{t^{4/5}}dt. \nonumber\\
\end{split}
\end{equation}
From the estimates (\ref{2.1}), (\ref{2.2}) and Cauchy's inequality we have
\begin{align*}
&\ \ \ \int_{0}^{T}|\zeta(13/20+3iy)|^{4}|\zeta(7/10-6iy)|dy \\
&\ll \bigg(\int_{0}^{T}|\zeta(13/20+3iy)|^{8}dy\bigg)^{1/2}\bigg(\int_{0}^{T}|\zeta(7/10-6iy)|^{2}dy\bigg)^{1/2} \\
&\ll T^{1+\varepsilon}.
\end{align*}
Let $L_2(v):=\int_0^v|\zeta(13/20+3iy)|^{4}|\zeta(7/10-6iy)|dy$. Using partial summation we obtain
\begin{align} \label{L 2 V}
&\ \ \ \int_{1}^V \frac{|\zeta(13/20+3iv)|^{4}|\zeta(7/10-6iv)|}{v^{4/5}}dv\nonumber\\
&=\int_{1}^V \frac{dL_2(v)}{v^{4/5}}\ll \frac{L_2(V)}{V^{4/5}}+\int_{1}^V \frac{L_2(v)}{v^{9/5}}\nonumber\\
&\ll V^{1/5+\varepsilon}+\int_{1}^V v^{-4/5+\varepsilon}dv\ll V^{1/5+\varepsilon}.
\end{align}
Hence
\begin{align} \label{H_{2}(x,T;3)}
H_{2}(x,T;3)\ll x^{11/20}T^{1/5+\varepsilon}.
\end{align}

Thus from (\ref{I(x,T;3)}), (\ref{residue 2/3}), (\ref{H_{1}(x,T;3)}), (\ref{H_{3}(x,T;3)}) and (\ref{H_{2}(x,T;3)}) we obtain
\begin{equation} \label{I(x,T;3) 3}
I(x,T;3)=x^{2/3}\sum_{j=0}^{4}a_{j}\log^{j}x+O(x^{5/8+\varepsilon})
\end{equation}
by choosing $T=x^{3/8}$, where $a_{j} \ (0\leq j\leq 4)$ are absolute constants.

\section*{Acknowledgements}
The authors would like to appreciate the referee for his/her patience in refereeing this paper. This work is supported by the National Natural Science Foundation of China (Grant Nos. 12301006, 12471009) and Beijing Natural Science Foundation (Grant No. 1242003).


\begin{thebibliography}{HD}


\normalsize
\baselineskip=17pt



\bibitem{H} M.~Hampejs, N.~Holighaus, L. T\'{o}th and C.~Wiesmeyr, {\it Representing and counting the subgroups of the group $\mathbb{Z}_{m}\times\mathbb{Z}_{n}$}, J. Numbers, Article ID 491428 (2014).

\bibitem{DR} D.~R.~Heath-Brown, \emph{The Pjateckii-Sapiro Prime Number Theorem}, J. Number Theory {16} (1983), 242--266.

\bibitem{Iv} A.Ivi\'{c}, \emph{The Riemann Zeta-Function. Theory and Applications},  Wiley,  New York, 1985.

\bibitem{NT} W.~G.~Nowak and L.~T\'{o}th, \emph{On the average number of subgroups of the group ${\Bbb Z}_{m}\times {\Bbb Z}_{n}$},
Int. J. Number Theory {10} (2014), 363--374.

\bibitem{PP} Chendong~Pan and Chenbiao~Pan, \emph{Foundations of Analytic Number Theory} (in Chinese),
Science Press: Beijing, 1990.

\bibitem{SL} Y. Sui and D. Liu, \emph{On the average number of subgroups of the groups ${\Bbb Z}_{m}\times {\Bbb Z}_{n}$ with $mn\leq x$},
 J. Number Theory {216} (2020), 264--279.


\bibitem{TZh} L. T\'{o}th and W.~G.~Zhai, {\it On the error term concerning the number of subgroups of the group $\mathbb{Z}_{m}\times\mathbb{Z}_{n}$ with $m, n \leq x$},  Acta Arith. {\bf 183}(3) (2018), 285--299.


\end{thebibliography}
\end{document}